\documentclass[12pt,a4paper]{article}
\usepackage{fullpage}
\usepackage[latin1]{inputenc}
\usepackage{amsmath,amssymb,srcltx}
\usepackage{amsmath,amssymb,amsthm,amstext} 
\usepackage{stmaryrd}
\usepackage{bbm}
\usepackage{yfonts}
\usepackage{tikz}
\usepackage[babel,german=quotes]{csquotes}

\newcommand{\bt}{\begin{thr}{\bf Theorem. }} 
\newcommand{\satz}{\begin{thr}{\bf Theorem. }\rm} 
\newcommand\E{\mathbb{E}}
\newcommand\R{\mathbb{R}}
\newcommand\p{\mathbb{P}}
\newcommand\N{\mathbb{N}}

\newcommand\D{\mathcal{D}}

\newcommand\F{\mathcal{F}}
\newcommand\C{\mathcal{C}}

\newcommand\A{\mathcal{A}}
\newcommand\B{\mathcal{B}}

\newcommand\ofp{\Omega,\mathcal{F},\mathbb{P}}
\renewcommand\i{\infty}

\newcommand\la{\lambda}
\newcommand\ve{\varepsilon}

\newcommand{\vp}{\varphi}

\DeclareMathOperator{\Var}{Var}
\newtheorem{theorem}{Theorem}[section]  % SECTION

\newtheorem{definition}[theorem]{Definition}
\newtheorem{lemma}[theorem]{Lemma}
\newtheorem{proposition}[theorem]{Proposition}

\newtheorem{meta-theorem}[equation]{Meta-theorem}
\newtheorem*{acka}{Acknowledgement}

\newtheorem{remark}[theorem]{Remark}

\title{The super-replication theorem\\
under proportional transaction costs revisited}

\author{Walter Schachermayer\footnote{Fakult\"at f\"ur Mathematik, Universit\"at Wien, Oskar-Morgenstern-Platz 1, A-1090 Wien, {\tt walter.schachermayer@univie.ac.at}. Partially supported by the Austrian Science Fund (FWF) under grant P25815 and Doktoratskolleg W1245, and the European Research Council (ERC) under grant FA506041.}\\
\vspace{3mm}\\
{\it dedicated to Ivar Ekeland on the occasion of his seventieth birthday}}

\begin{document}

\date{\today}
\maketitle

%\begin{center} \large{Preliminary version} \\

%\Large{!!!! DO NOT CIRCULATE !!!!}
%\end{center}

\begin{abstract}
We consider a financial market with one riskless and one risky asset. The super-replication theorem states that there is no duality gap in the problem of super-replicating a contingent claim under transaction costs and the associated dual problem. We give two versions of this theorem.

The first theorem relates a numéraire-based admissibility condition in the primal problem to the notion of a local martingale in the dual problem. The second theorem relates a numéraire-free admissibility condition in the primal problem to the notion of a uniformly integrable martingale in the dual problem.
\end{abstract}

\section{Introduction}
The essence of the Black-Scholes theory (\cite{BS73}, \cite{M73}) goes as follows: in the framework of their model $S=(S_t)_{0 \leq t \leq T}$ of a financial market (with riskless interest rate $r$ normalized to $r=0$) the unique arbitrage-free price for a contingent claim $X_T$ maturing at time $T$ is given by
\begin{equation}\label{I2}
X_0= \mathbb{E}_{Q} [X_T]. 
\end{equation}

Here $Q$ is the ``martingale measure" for the Black-Scholes model, i.e.~the probability measure on $(\Omega, \mathcal{F}_T, \mathbb{P})$ under which $S$ is a martingale.
The paper of Harrison-Kreps \cite{HK79} marked the beginning of a deeper understanding of the notion of arbitrage and its relation to martingale theory. Today it is very well understood that the salient feature of the Black-Scholes model which causes \eqref{I2} to yield the unique arbitrage-free price is the fact that the martingale measure $Q$ is {\it unique} in this model.

Financial markets $S$ admitting a unique martingale measure $Q$ are called "complete financial markets". We remark in passing that in this informal introduction we leave technicalities aside, such as integrability assumptions or the requirement that this measure $Q$ should be {\it equivalent} to the original measure $\mathbb{P}$, i.e.~$Q [A]=0$ if and only if $\mathbb{P}[A]=0.$

In a complete market $S=(S_t)_{0 \leq t \leq T}$ every contingent claim $X_T$ can be {\it perfectly replicated}, i.e.~there is a predictable process $H=(H_t)_{0 \leq t \leq T}$ such that 
\begin{equation}\label{I4a}
X_T = X_0 + \int^T_0 H_t dS_t.
\end{equation}
We now pass to the more realistic setting of a possibly {\it incomplete} financial market $S=(S_t)_{0 \leq t \leq T}.$ By definition we assume that the set $\mathcal{M}^{e}(S)$ of equivalent martingale measures is non-empty, but (possibly) not reduced to a singleton. In this setting the valuation formula \eqref{I2} is replaced by
\begin{equation}\label{I4}
X_0=\sup_{Q \in \mathcal{M}^{e}(S)} \mathbb{E}_Q[X_T]
\end{equation}

This real number $X_0$ is called the {\it super-replication price} of $X_T$. The reason for this name is that one may find a predictable strategy $H=(H_t)_{0 \leq t \leq T}$ such that the equality \eqref{I4a} now is replaced by the inequality
\begin{equation}\label{I5}
X_T \leq X_0 + \int_0^T H_t dS_t
\end{equation}
and $X_0$ is the smallest number with this property. This is the message of the {\it super-replication theorem} which was established by N.~El Karoui and M.-C.~Quenez \cite{EQ95} in a Brownian framework and, in greater generality, by F.~Delbaen and the author in \cite{DS94} (compare \cite{DS06} for a comprehensive account).
\vskip10pt
The theme of the present paper is to show (two versions of) a super-replication theorem in the presence of transaction costs $\lambda >0.$ For a given financial market $S=(S_t)_{0 \leq t \leq T}$ as above we now suppose that we can {\it buy} the stock at price $S$ but can only {\it sell} it at price $(1-\lambda)S.$ The higher price $S$ is called the {\it ask price} while the lower price $(1-\lambda)S$ is called the {\it bid price}.

In this context the notion of {\it martingale measures} $Q$ appearing in \eqref{I4} is replaced by the following concept which goes back to the pioneering work of E.~Jouini and H.~Kallal \cite{JK95}.

\begin{definition}\label{1.1}
Fix a price process $S=(S_t)_{0 \leq t \leq T}$ and transaction costs $0 < \lambda <1$ as above. A {\it consistent price system} (resp.~a consistent  local price system) is a pair $(\widetilde{S}, Q)$ such that $Q$ is a probability measure equivalent to $\mathbb{P}$ and $\widetilde{S}=(\widetilde{S}_t)_{0 \leq t \leq T}$ takes its values in the bid-ask spread $[(1-\lambda) S,S]=([(1-\lambda)S_t, S_t])_{0 \leq t \leq T}$  and $\widetilde{S}$ is a $Q$-martingale (resp.~a local $Q$-martingale). 

To stress the difference of the two notions we shall sometimes call a consistent price system a consistent price system in the non-local sense.
\end{definition}

The condition of the {\it existence of an equivalent martingale measure} in the frictionless setting corresponds to the following notion.

\begin{definition}\label{1.2}
For $0 < \lambda < 1,$ we say that a price process $S=(S_t)_{0 \leq t \leq T}$ satisfies $(CPS^\lambda)$ (resp.~$(CPS^\lambda)$ in a local sense) if there exists a consistent price system (resp.~a consistent local price system).
\end{definition}

It is the purpose of this article to identify the precise assumptions in order to establish an analogue to \eqref{I4} and \eqref{I5} above, after translating these statements into the context of financial markets under transaction costs. To make concrete what we have in mind, we formulate our program in terms of a not yet precisely formulated ``meta-theorem".

\begin{theorem}\label{t1.3}
(not yet precise version of super-hedging) Fix a financial market $S=(S_t)_{0 \leq t \leq T}$, transaction costs $0 < \lambda < 1$, and a contingent claim which pays $X_T$ many units of bond at time $T$. Assume that $S$ satisfies an appropriate regularity condition (of no arbitrage type). For a number $X_0 \in \mathbb{R}$, the following assertions are equivalent.
\begin{itemize}
\item[(i)] $X_T$ can be {\it super-replicated} by starting with an initial portfolio of $X_0$ many units of bond and subsequently trading in $S$ under transaction costs $\lambda$. The trading strategy has to be {\it admissible} in an appropriate sense.
\item[(ii)] For every consistent price system $(\widetilde{S}, Q)$ (in an appropriate sense, i.e.~local or global) we have
\begin{equation*}
X_0 \geq \mathbb{E}_Q[X_T]. 
\end{equation*}
\end{itemize}
\end{theorem}

We shall formulate below two versions which turn the above ``meta-theorem" into precise mathematical statements. Let us first comment on the history of the above result. E.~Jouini and H.~Kallal in their pioneering paper \cite{JK95} considered a Hilbert space setting and proved a version of the above theorem in this context. They have thus established a perfect equivalent to the paper \cite{HK79} of Harrison-Kreps, replacing the frictionless theory by a model involving proportional transaction costs.

Y.~Kabanov \cite{K99} proposed a numéraire-free setting of multi-currency markets (see \cite{KS09}) for more detailed information) which is much more general than the present setting. In \cite{KS02} Y.~Kabanov and Ch.~Stricker proved a version of the super-hedging theorem in Kabanov's model under the assumption of continuity of the exchange rate processes. This continuity assumption was removed by L.~Campi and the author in \cite{CS06} thus establishing a general version of the super-hedging theorem in Kabanov's framework. 
However, due to the generality of the model considered in \cite{K99}, \cite{KS02}, and \cite{CS06}, the precise definitions of e.g.~self-financing portfolios and admissibility are sometimes difficult to check in applications.

We therefore change the focus in the present paper and concentrate on a more concrete setting with just one stock and one (normalised) bond, as well as fixed transaction costs $\lambda >0.$ Our aim is to establish clear-cut and easy-to-apply versions of the above super-hedging ``meta-theorem" \ref{t1.3}. Most importantly, we shall clarify the difference between a numéraire-free and a numéraire-based notion of admissible portfolios and its correspondence to the concepts of martingales and local martingales. This is somewhat analogue to the "numéraire-free" and "numéraire-based" versions of the Fundamental Theorem of Asset Pricing under Transaction Costs established in \cite{GRS10}. In the frictionless setting, analogous results are due to J.~Yan \cite{Y05} (compare also \cite{DS95}, and \cite{Y98}).
\vskip10pt
We now state the two versions of the super-hedging theorem which we shall prove in this paper. The terms appearing in the statements will be carefully defined in the next section.

\begin{theorem}[numéraire-based super-hedging]\label{t1.4}
%Let $S=(S_t)_{0 \leq t \leq T}$ be a {\it locally bounded} c\`adl\`ag process, defined on and adapted to a filtered stochastic base $(\Omega, \mathcal{F}, (\mathcal{F}_t)_{0 \leq t \leq T}, \mathbb{P})$. Fix $0 < \lambda < 1$ and assume that, for every $0 < \lambda' < \lambda$ there is a consistent price system. Let $X_T$ be an $\mathcal{F}_T$-measurable random variable such that
%$$X_T \geq - M$$
%for some $M \geq 0$. For a number $X_0 \in \mathbb{R}$, the following assertions are equivalent.

Fix an $\mathbb{R}_+$-valued adapted c\`adl\`ag process $S=(S_t)_{0 \leq t \leq T}$, transaction costs $0 < \lambda < 1$, and a contingent claim which pays $X_T$ many units of bond at time $T$. The random variable $X_T$ is assumed to be uniformly bounded from below. Assume that, for each $0 < \lambda' < 1$, the process $S$ satisfies $(CPS^{\lambda'})$ in a local sense. For a number $X_0 \in \mathbb{R}$, the following assertions are equivalent.
\begin{itemize}
\item [(i)] There is a self-financing trading strategy $\varphi=(\varphi^0_t, \varphi^1_t)_{0 \leq t \leq T}$ such that 
$$\varphi_0=(X_0, 0) \quad \mbox{and} \quad \varphi_T=(X_T,0)$$
which is admissible in the following numéraire-based sense: there is $M \geq 0$ such that, for every $[0,T]$-valued stopping time $\tau,$
\begin{equation}\label{D3.1}
V_\tau(\varphi) \geq -M, \qquad \mbox{a.s.}
\end{equation}
\item [(ii)] For every consistent local price system, i.e.~for every probability measure $Q$, equivalent to $\mathbb{P}$, such that there is a {\it local martingale} $\widetilde{S}=(S_t)_{0 \leq t \leq T}$ under $Q$, taking its values in the bid-ask spread $[(1-\lambda) S,S]=([(1-\lambda) S_t, S_t])_{0 \leq t \leq T}$, we have
\begin{equation}\label{zp4}
X_0 \geq \mathbb{E}_Q[X_T].
\end{equation}
\end{itemize}
\end{theorem}

\begin{theorem}[numéraire-free super-hedging]\label{t1.5}
%Let $S=(S_t)_{0 \leq t \leq T}$ be a c\`adl\`ag process, defined on and adapted to a filtered stochastic base $(\Omega, \mathcal{F}, (\mathcal{F}_t)_{0 \leq t \leq T}, \mathbb{P})$. Fix $0 < \lambda < 1$ and assume that, for every $0 < \lambda' < \lambda$ there is a consistent price system. Let $X_T$ be an $\mathcal{F}_T$-measurable random variable such that
%$$X_T \geq - M(1+S_T)$$
%for some $M \geq 0$. For a number $X_0 \in \mathbb{R}$, the following assertions are equivalent.

Fix an $\mathbb{R}_+$-valued adapted c\`adl\`ag process $S=(S_t)_{0 \leq t \leq T}$, transaction costs $0 < \lambda < 1$, and consider a non-negative contingent claim which pays $X_T$ many units of bond at time $T$. The random variable $X_T$ is assumed to be bounded from below by a multiple of $(1+S_T).$ Assume that, for each $0 < \lambda' < 1$ the process $S$ satisfies $(CPS^{\lambda'})$ in a non-local sense. For a number $X_0 \in \mathbb{R}$, the following assertions are equivalent.

\begin{itemize}
\item [(i)] There is a self-financing trading strategy $\varphi=(\varphi^0_t, \varphi^1_t)_{0 \leq t \leq T}$ such that 
$$\varphi_0=(X_0, 0) \quad \mbox{and} \quad \varphi_T= (X_T, 0)$$
which is admissible in the following sense: there is $M \geq 0$ such that, for every $[0,T]$-valued stopping time $\tau,$ 
\begin{equation}\label{D4}
V_{\tau}(\varphi) \geq -M(1+S_{\tau}), \quad \mbox{a.s.}
\end{equation}
\item [(ii)] For every consistent price system, i.e.~for every probability measure $Q$, equivalent to $\mathbb{P}$, such that there is a {\it martingale} $\widetilde{S}=(S_t)_{0 \leq t \leq T}$ under $Q$, taking its values in the bid-ask spread $[(1-\lambda) S,S]=([(1-\lambda) S_t, S_t])_{0 \leq t \leq T}$ we have
\begin{equation}\label{zp4a}
X_0 \geq \mathbb{E}_Q[X_T]
\end{equation}
\end{itemize}
\end{theorem}

Why do we speak about "numéraire-based" and "numéraire-free"? The admissibility condition of Theorem \ref{t1.4} refers to the bond as numéraire. Condition \eqref{D3.1} means that an agent can cover the trading strategy $\varphi$ by holding $M$ units of bond. In contrast, condition \eqref{D4} means that an agent can cover the trading strategy $\varphi$ by holding $M$ units of bond as well as $M$ units of stock. The latter assumption is symmetric between stock and bond. It does not single out one asset as numéraire and is therefore called "numéraire-free".

\section{Definitions and Notations}
%We fix a filtered stochastic probability space $(\Omega, \mathcal{F}, \mathcal{F_t}_{0 \leq t \leq T}, \mathbb{P})$, satisfying the usual conditions of right continuity and completeness, as well as an adapted c\`adl\`ag process $S=(S_t)_{0 \leq t \leq T}$. The random variable $S_t$ models the price of a stock in terms of a bond which is normalised to $B_t \eqiv 1.$ We also fix transaction costs $0 < \lambda < 1$.

We consider a financial market consisting of one riskless asset and one risky asset. The riskless asset has constant price $1$ and can be traded without transaction cost. The price of the risky asset is given by a strictly positive adapted c\`adl\`ag stochastic process $S=(S_t)_{0 \leq t \leq T}$ on some underlying filtered probability space $\big(\Omega, \mathcal{F}, (\mathcal{F}_t)_{0 \leq t \leq T}, P\big)$ satisfying the usual assumptions of right continuity and completeness. In addition, we assume that $\mathcal{F}_0$ is trivial. For technical reasons (compare \cite{CS06}) we also assume (w.l.g.) that $\mathcal{F}_T= \mathcal{F}_{T-}$ and $S_T=S_{T-}.$

\emph{Trading strategies} are modeled by $\R^2$-valued, predictable processes $\vp=(\vp^0_t,\vp^1_t)_{0\leq t\leq T}$ of finite variation, where $\vp^0_{t}$ and $\vp^1_{t}$ denote the holdings in units of the riskless and the risky asset, respectively, after rebalancing the portfolio at time $t$. For any process $X$ of finite variation we denote by $X=X_0+X^{\uparrow}-X^{\downarrow}$ its Jordan-Hahn decomposition into two non-decreasing processes $X^{\uparrow}$ and $X^{\downarrow}$ both null at zero. The total variation $\Var_t(X)$ of $X$ on $[0,t]$ is then given by $\Var_t(X)=X^{\uparrow}_t+X^{\downarrow}_t$ and the continuous part $X^c$ of $X$ by
$$X^c_t:=X_t-\sum_{s<t} \Delta_+ X_s -  \sum_{s\leq t} \Delta X_s,$$
where $\Delta_+ X_t:=X_{t+}-X_t$ and $\Delta X_t:=X_t-X_{t-}$. Trading in the risky asset incurs proportional transaction costs of size $\lambda \in (0,1).$ This means that one has to pay a (higher) ask price $S_t$ when buying risky shares at time $t$ but only receives a (lower) bid price $(1-\lambda)S_t$ when selling them. 

A strategy $\vp=(\vp^0_t,\vp^1_t)_{0\leq t\leq T}$ is called \emph{self-financing under transaction costs $\lambda$} if
\begin{align}\label{sfc}
\int^t_s d\varphi^0_u \leq - \int^t_s S_u d\varphi^{1,\uparrow}_u + \int^t_s(1-\lambda)S_u d\varphi^{1,\downarrow}_u
\end{align}
a.s.~for all $0 \leq s < t \leq T$, where
\begin{align*}
\int^t_s S_u d\varphi^{1,\uparrow}_u&:= \int^t_s S_u d\varphi_u^{1,\uparrow,c} + \sum_{s < u \leq t}S_{u-}\Delta \varphi_u^{1,\uparrow} + \sum_{s \leq u < t} S_u \Delta_+ \varphi_u^{1,\uparrow},\\
\int^t_s(1-\lambda) S_u d\varphi_u^{1,\downarrow}&:= \int^t_s (1-\lambda) S_u d\varphi_u^{1, \downarrow,c} + \sum_{s < u \leq t} (1-\lambda)S_{u-}\Delta\varphi^{1, \downarrow}_u + \sum_{s \leq u < t} (1-\lambda)S_u \Delta_+ \varphi^{1,\downarrow}_u
\end{align*}
can be defined by using Riemann-Stieltjes integrals, as $S$ is c\`adl\`ag.
The self-financing condition \eqref{sfc} then states that purchases and sales of the risky asset are accounted for in the riskless position:

\begin{align}%\label{eq:sf2}
d\varphi^{0,c}_t&\leq-S_td\varphi^{1,\uparrow,c}_t+(1-\lambda)S_td\varphi^{1,\downarrow,c}_t, \quad &0 \leq t \leq T, \label{sf2.1}\\
\Delta\varphi^0_t&\leq-S_{t-}\Delta\varphi^{1,\uparrow}_t +(1-\lambda)S_{t-}\Delta\varphi^{1,\downarrow}_t, \quad &0 \leq t \leq T, \label{sf2.2} \\
\Delta_+\varphi^0_t&\leq-S_t\Delta_+\varphi^{1,\uparrow}_t +(1-\lambda)S_t\Delta_+\varphi^{1,\downarrow}_t, \quad &0 \leq t \leq T. \label{sf2.3}
\end{align}

We define the \emph{liquidation value} at time $t$ by
\begin{align}
V_t(\vp):={}&\vp^0_t+(\vp^1_t)^+(1-\lambda) S_t-(\vp^1_t)^-S_t.
\end{align}

We have the following two notions of admissibility

\begin{definition}\label{d2.1}
\begin{itemize}
\item [(a)] A self-financing trading strategy $\varphi$ is called admissible in a numéraire-based sense if there is $M>0$ such that, for every $[0,T]$-valued stopping time $\tau$, 
\begin{equation}\label{p5}
V_\tau(\varphi) \geq -M, \quad \mbox{a.s.,}
\end{equation}

\item [(b)] A self-financing trading strategy $\varphi$ is called admissible in a numéraire-free sense if there is $M>0$ such that, for every $[0,T]$-valued stopping time $\tau$, 
\begin{equation}\label{p5a}
V_t(\varphi) \geq -M(1+S_t), \quad \mbox{a.s.}
\end{equation}

\end{itemize}
\end{definition}
\vskip10pt
\noindent
Here are typical examples of self-financing trading strategies.
\vskip10pt
\begin{definition}\label{def4.7}
Fix $S$ and $\la >0,$ as above, let $\tau :\Omega \to [0,T] \cup \{\i\}$ be a stopping time, and let $f_\tau ,g_\tau$ be $\F_{\tau}$-measurable $\R_+$-valued functions. We define the
corresponding \textnormal{ask} and \textnormal{bid} processes as
\begin{align}
a_t &=  (-S_\tau ,1)f_\tau ~  \mathbbm{1}_{\rrbracket \tau ,T\rrbracket} (t),  &0\le t\le T, \label{155}  \\
b_t &=  ((1-\la)S_\tau, -1) g_\tau  ~ \mathbbm{1}_{\rrbracket \tau ,T\rrbracket}(t),  &0\le t\le T. \label{156}
\end{align}

Similarly, let $\tau:\Omega \to [0,T] \cup \{\infty \}$ be a predictable stopping time, and let $f_\tau ,g_\tau$ be $\F_{\tau-}$-measurable $\R_+$-valued functions. We define
\begin{align}
a_t &=  (-S_{\tau-} ,1)f_\tau ~  \mathbbm{1}_{\llbracket \tau ,T\rrbracket} (t),  &0\le t\le T, \label{B1}  \\
b_t &=  ((1-\la)S_{\tau-},-1) g_\tau  ~ \mathbbm{1}_{\llbracket \tau ,T\rrbracket}(t),  &0\le t\le T. \label{B1a}
\end{align}

We call a process $\varphi=(\varphi^0_t ,\varphi^1_t)_{0\le t\le T}$ a \textnormal{predictable, simple, self-financing} process, if it is a finite sum of ask and bid processes as above. 
\end{definition}

We note that $a=(a^0_t, a^1_t)_{0 \leq t \leq T}$ defined in \eqref{155} is admissible (in either sense of the above definitions) if the random variable $f_\tau S_\tau$ is bounded from above. As regards $b=(b^0_t, b^1_t)_{0 \leq t \leq T}$ defined in \eqref{156} it is admissible in the numéraire-free sense if $g_\tau$ is bounded; it is admissible in the numéraire-based sense if the process $(g_\tau S_t)_{\tau < t \leq T}$ is uniformly bounded. 

Analogous remarks apply to \eqref{B1} and \eqref{B1a}.
\vskip10pt

\section{Closedness in measure}
The following lemma was proved by L.~Campi and the author in the general framework of Kabanov's modeling of $d$-dimensional currency markets. Here we adapt
the proof for a single risky asset model. 

In section 2 we postulated as a qualitative --- a priori --- assumption that the strategies
$\varphi=(\varphi^0,\varphi^1)$ have {\it finite variation}. The next lemma provides an automatic --- a posteriori --- quantitative control on the size of the finite variation. Note that we make a combination of the weaker versions of our hypotheses: as regards the no-arbitrage type assumption we only suppose $(CPS^{\lambda'})$ in the local sense and as regards admissibility we only require it in the numéraire-free sense.

\begin{lemma}\label{l4.9}
Let $S$ and $0 < \la <1$ be as above, and suppose that $(CPS^{\la '})$ is satisfied in the local sense, for some $0<\la '<\la$. Fix $M>0.$ Then the total variation of the process $(\varphi^0_t,\varphi^1_t)_{0\le t\le T}$ remains bounded in $L^0(\ofp)$, when $\varphi=(\varphi^0,\varphi^1)$ 
runs through all $M$-admissible $\lambda$-self-financing strategies  (in the numéraire-free sense \eqref{p5a}).

More explicitly: for $M>0$ and $\ve >0$, there is $C>0$ such that, for all $M$-admissible, $\lambda$-self-financing strategies $(\varphi^0,\varphi^1),$ starting at
$(\varphi^0_{0},\varphi^1_{0})=(0,0),$ and all increasing sequences
$0=\tau_0 < \tau_1 <\ldots <\tau_K=T$ of stopping times we have
\begin{align}\label{165}
&\p\left[\sum\limits^K_{k=1} \ |\varphi^0_{\tau_k} -\varphi^0_{\tau_{k-1}}| \ \geq C\right] <\ve,  \\
&\p\left[\sum\limits^K_{k=1} \ |\varphi^1_{\tau_k} -\varphi^1_{\tau_{k-1}}| \ \geq C\right] <\ve. \label{166} \\ \nonumber
\end{align}
\end{lemma}

\underline{Proof:}
Fix $0<\la '<\la$ as above. By hypothesis there is a probability measure $Q \sim \p,$ and a local $Q$-martingale $(\widetilde{S}_t)_{0\le t\le T}$ such that
$\widetilde{S}_t\in [(1-\la ')S_t,S_t].$ As the assertion of the lemma is of local type we may assume, by stopping, that $\widetilde{S}$ is a true martingale. We also may assume w.l.g.~that $\varphi^1_T=0$, i.e., that the  position in stock is liquidated at time $T$.

Fix $M>0$ and a $\lambda$-self-financing, $M$-admissible (in the sense of \eqref{p5a}) process $(\varphi^0_t,\varphi^1_t)_{t\geq 0},$ starting at 
$(\varphi^0_{0},\varphi^1_{0})=(0,0).$ Write $\varphi^0 =\varphi^{0,\uparrow} -\varphi^{0,\downarrow}$ and $\varphi^1 =\varphi^{1,\uparrow}-\varphi^{1,\downarrow}$ as the 
canonical differences of increasing processes. We shall show that
\begin{equation}\label{p10}
\mathbb{E}_Q[\varphi_T^{0,\uparrow}] \leq \frac{M(1+\mathbb{E}_Q[S_T])}{\lambda-\lambda'}
\end{equation}
Define the process $\varphi'=((\varphi^0)',(\varphi^1)')$ by
\begin{align*}
\varphi'_t=\left((\varphi^0)'_t,(\varphi^1)'_t\right)=\left(\varphi^0_t +\frac{\la-\la'}{1-\la} \varphi^{0,\uparrow}_t,\varphi^1_t\right), \qquad\qquad {0\le t\le T}.
\end{align*}
This is a self-financing process under transaction costs $\la'$: indeed, whenever $d\varphi^0_t >0$ so that $d\varphi^0_t=d\varphi^{0,\uparrow}_t,$ the agent sells stock and receives 
$d\varphi^{0,\uparrow}_t=(1-\la)S_t d\varphi_t^{1,\downarrow}$ (resp. $(1-\la') S_t d\varphi_t^{1,\downarrow} = \tfrac{1-\la'}{1-\la} d\varphi^{0,\uparrow}_t$) many bonds under transaction costs
$\la$ (resp.~$\la'$). The difference between these two terms is $\tfrac{\la-\la'}{1 -\la} d\varphi^{0,\uparrow}_t$; this is the amount by which the $\la'$-agent does better than
the $\la$-agent. It is also clear that $((\varphi^0)',(\varphi^1)')$ under transaction costs $\la'$ still is a $M$-admissible strategy (in the numéraire-free sense of \eqref{p5a}).

By Proposition 2.3 of \cite{S13} the process 
$$((\varphi^0)'_t + (\varphi^1)'_t \widetilde{S}_t)_{0 \leq t \leq T}=((\varphi^0)'_t+\varphi^1_t\tilde{S}_t)_{0\le t\le T} =(\varphi^0_t+\tfrac{\la-\la'}{1-\la}\varphi^{0,\uparrow}_t+ \varphi^1_t\tilde{S}_t)_{0\le t\le T}$$
is an optional strong  $Q$-super-martingale. Hence
\begin{equation}\label{D6}
\E_Q[\varphi^0_T+\varphi^1_T\tilde{S}_T]+\tfrac{\la-\la'}{1-\la}\E_Q[\varphi_T^{0,\uparrow}]\le 0.
\end{equation}
As 
\begin{equation}\label{P10}
\varphi^0_T=\varphi^0_T + \varphi^1_T \widetilde{S}_T \geq -M(1 + S_T).
\end{equation}
we have shown \eqref{p10}.

\medskip

To obtain a control on $\varphi^{0,\downarrow}_T$ too, note that $\varphi^0_T=\varphi^0_T+\varphi^1_T\widetilde{S}_T \geq -M(1+S_T)$ as $\varphi^1_T=0$ so that $\varphi^{0,\downarrow}_T\le \varphi^{0,\uparrow}_T+M(1+S_T).$ 
Therefore we obtain the following estimate for the total variation $\varphi^{0,\uparrow}_T+\varphi^{0,\downarrow}_T$ of $\varphi^0$
\begin{equation}\label{168}
\E_Q\left[\varphi^{0,\uparrow}_T+\varphi^{0,\downarrow}_T\right]\le M\left(\frac{2}{\la-\la'}+1\right) \bigg ( 1+\mathbb{E}_Q[S_T] \bigg).
\end{equation}

The passage from the $L^1(Q)$-estimate \eqref{168} to the $L^0(\p)$-estimate \eqref{165} is standard: for $\ve >0$ there is $\delta >0$ such that for subsets
$A\in\mathcal{F}$ with $Q[A] <\delta$ we have
$\p[A] <\ve.$ Letting $C=\tfrac{M}{\delta}(\tfrac{2}{\la-\la'} +1) (1+\mathbb{E}_Q[S_T])$ and applying Tschebyscheff to \eqref{168} we get
\begin{equation}\label{p8}
\p\left[\varphi^{0,\uparrow}_T+\varphi^{0,\downarrow}_T\geq C\right] <\ve,
\end{equation}
which implies \eqref{165}.

As regards \eqref{166} it follows from \eqref{sfc} that 
\begin{equation}\label{J1}
d\varphi^{1, \uparrow}_t \leq \frac{d \varphi^{0,\downarrow}_t}{S_t},
\end{equation}
or, more precisely, by \eqref{sf2.1}, \eqref{sf2.2}, and \eqref{sf2.3},
\begin{equation}\label{J2}
d\varphi^{1, \uparrow,c}_t \leq \frac{d \varphi^{0,\downarrow,c}_t}{S_t},
\end{equation}
\begin{equation}\label{J3}
\Delta\varphi^{1, \uparrow}_t \leq \frac{\Delta \varphi^{0,\downarrow}_t}{S_{t-}},
\end{equation}
\begin{equation}\label{J4}
\Delta_+\varphi^{1, \uparrow}_t \leq \frac{\Delta_+ \varphi^{0,\downarrow}_t}{S_t}.
\end{equation}

By assumption the trajectories of $(S_t)_{0 \leq t \leq T}$ are strictly positive. In fact, we even have, for almost all trajectories $(S_t(\omega))_{0 \leq t \leq T}$, that $\inf_{0 \leq t \leq T} S_t(\omega)$ is strictly positive. Indeed, $\widetilde{S}$ being a $Q$-martingale with $\widetilde{S}_T > 0$ a.s.~satisfies that $\inf_{0 \leq t \leq T} \widetilde{S}_t(\omega)$ is $Q$-a.s.~and therefore $\mathbb{P}$-a.s.~strictly positive.

Summing up, for $\ve >0$, we may find $\delta >0$ such that
\begin{align*}
\p\left[\inf\limits_{0\le t\le T} S_t <\delta\right] <\frac{\ve}{2}.
\end{align*}

Hence we may control $\varphi_T^{1,\uparrow}$ by using \eqref{J1} and estimating $\varphi^{0,\downarrow}$ by \eqref{p8}. Finally, we can control $\varphi_T^{1,\downarrow}$ by simply observing that $\varphi^{1, \uparrow}_T - \varphi^{1,\downarrow}_T=\varphi^1_T - \varphi^1_0=0.$
\hfill\vrule height6pt width6pt depth0pt \medskip

\begin{remark}\label{r4.10}
\textnormal{
In the above proof we have shown that the elements $\nobreak{\varphi^{0,\uparrow}_T, \varphi^{0,\downarrow}_T, \varphi^{1,\uparrow}_T, \newline \varphi^{1,\downarrow}_T}$ remain bounded in $L^0(\ofp),$ when $(\varphi^0,\varphi^1)$ runs through the $M$-admissible (in the numéraire-free sense \eqref{p5a}) self-financing processes and $\varphi^0=\varphi^{0,\uparrow} -\varphi^{0,\downarrow}$ and
$\varphi^1=\varphi^{1,\uparrow}-\varphi^{1,\downarrow}$ denote the canonical decompositions. For later use we remark
that the proof shows, in fact, that also the convex combinations of the functions $\varphi^{0,\uparrow}_T$ etc.~remain
bounded in $L^0(\ofp).$ Indeed the estimate \eqref{p10} shows that the convex hull of the functions 
$\varphi^{0,\uparrow}_T$ is bounded in $L^1(Q)$ and \eqref{168} yields the same for $\varphi^{0,\downarrow}_T.$ For
$\varphi^{1,\uparrow}_T$ and $\varphi^{1,\downarrow}_T$ the argument is similar.}
\end{remark}
\vskip10pt

We can now formulate the main result of this section, in a numéraire-based as well as a numéraire-free version (Theorem \ref{t4.14} and Theorem \ref{t4.14.1})

\begin{definition}
For $M>0$ we denote by $\mathcal{A}_{nb}^M$ (resp.  $\mathcal{A}_{n\!f}^M$) the set of pairs $(\varphi^0_T, \varphi^1_T) \in L^0(\mathbb{R}^2)$ of terminal values of self-financing trading strategies $\varphi$, starting at $\varphi_0=(0,0)$, which are $M$-admissible in the numéraire-based sense \eqref{p5} (resp. in the numéraire-free sense \eqref{p5a}).

We denote by $\C^M_{nb}$ (resp.~$\C^M_{n\!f}$) the set of random variables $\varphi^0_T \in L^0$ such that $(\varphi^0_T, 0)$ is in $\mathcal{A}^M_{nb}$ (resp.~in $\mathcal{A}^M_{n\!f}$).

We shall occasionally drop the sub-scripts $nb$ (resp.~$n\!f$) when it is clear from the context that we are in the numéraire-based (resp.~numéraire-free) setting.
\end{definition}

\begin{theorem}\label{t4.14}
(numéraire-based version) Fix $S=(S_t)_{0\le t\le T}$ and $0 < \la < 1$ as above, and suppose that $(CPS^{\la'})$ is satisfied in a local sense, for each $0<\la'<\la$. For $M>0,$ the convex set $\A^M_{nb}\subseteq L^0(\ofp;\R^2)$ as well as the convex set $\C^M_{nb}\subseteq L^0(\ofp)$ are closed with respect to the topology of convergence in measure.
\end{theorem}

{ \underline{Proof}: }\
Fix $M>0$ and let $(\varphi^{n}_T)^\i_{n=1}=(\varphi^{0,n}_T,\varphi^{1,n}_T)^\i_{n=1}$ be a sequence in $\A^M=\A^M_{nb}$ converging a.s.~to some $\varphi_T=(\varphi^0_T, \varphi^1_T) \in L^0(\mathbb{R}^2).$ We have to show that $\varphi_T \in \mathcal{A}^M.$ We may find self-financing, admissible (in the numéraire-based sense) strategies 
$\varphi^n=(\varphi^{0,n}_t,\varphi^{1,n}_t)_{0\le t\le T},$ starting at
$(\varphi^{0,n}_{0},\varphi^{1,n}_{0})=(0,0),$ and ending with terminal values $(\varphi^{0,n}_T,\varphi^{1,n}_T).$ By the assumption $(CPS^{\lambda'})$, for {\it each} $0 < \lambda' < \lambda$, we may conclude that these processes are $M$-admissible in the numéraire-based sense (\cite{S13}, Th.~1.7). As above, decompose canonically these processes as $\varphi^{0,n}_T=\varphi_T^{0,n,\uparrow} - \varphi_T^{0,n,\downarrow},$ and $\varphi^{1,n}_T =\varphi^{1,n,\uparrow}_T 
-\varphi^{1,n,\downarrow}_T.$ By Lemma \ref{l4.9} and the subsequent remark we know that $(\varphi_T^{0,n,\uparrow})^\i_{n=1},(\varphi^{0,n,\downarrow}_T)^\i_{n=1}, (\varphi^{1,n,\uparrow}_T)^\i_{n=1},$ and 
$(\varphi^{1,n,\downarrow}_T)^\i_{n=1}$ as well as their convex combinations are bounded in $L^0(\ofp)$, so that by Lemma A1.1a in \cite{DS94} we may find convex combinations converging
a.s.~to elements $\varphi_T^{0,\uparrow}, \varphi_T^{0,\downarrow}, \varphi_T^{1,\uparrow},$ and $\varphi_T^{1,\downarrow} \in L^0(\ofp).$ To alleviate notation we denote these 
sequences of convex combinations still by the original sequences.
We claim that $(\varphi^0_T,\varphi^1_T)=(\varphi_T^{0,\uparrow} - \varphi_T^{0,\downarrow},\varphi_T^{1,\uparrow} -\varphi_T^{1,\downarrow})$ is in $\A^M$ which will readily show 
the closedness of $\A^M$ with respect to the topology of convergence in measure.

By inductively passing to convex combinations, still denoted by the original sequences, we may, for each rational number $r\in[0,T[$, assume that $(\varphi^{0,n,\uparrow}_r)^\i_{n=1},
(\varphi^{0,n,\downarrow}_r)^\i_{n=1}, (\varphi^{1,n,\uparrow}_r)^\i_{n=1},$ and $(\varphi^{1,n,\downarrow}_r)^\i_{n=1}$ converge to some elements 
$\bar{\varphi}^{0,\uparrow}_r , \bar{\varphi}^{0,\downarrow}_r, \bar{\varphi}^{1,\uparrow}_r,$ and $\bar{\varphi}^{1,\downarrow}_r$ in $L^0(\ofp).$ 
By passing to a diagonal subsequence, we may suppose that this convergence holds true for all rationals $r\in[0,T[.$

Clearly the four processes $\bar{\varphi}^{0,\uparrow}_{r\in \mathbb{Q}\cap [0,T[}$ etc, indexed by the rationals $r$ in $[0,T[,$ still are a.s.~increasing and define an $M$-admissible 
process in the numéraire-based sense of \eqref{p5}, indexed by $[0,T[\cap \mathbb{Q}$. 

We have to extend these processes to all real numbers $t\in [0,T].$ This is done by first letting
\begin{equation}\label{178}
\widehat{\varphi}_t^{0,\uparrow} =\lim\limits_{\substack{ r\searrow t \\ r\in \mathbb{Q}}} \bar{\varphi}^{0,\uparrow}_r, \qquad 0\le t <T,
\end{equation}
and $\widehat{\varphi}^{0,\uparrow}_{0} =0.$ The terminal value $\widehat{\varphi}_T^{0,\uparrow} = \varphi_T^{0,\uparrow}$ is still given by the first step of the construction. The c\`adl\`ag process $\widehat{\varphi}^{0,\uparrow}$ is not yet the desired limit as we  still have to take special care of the jumps of $\widehat{\varphi}^{0,\uparrow}$. The jumps of the process $\widehat{\varphi}^{0,\uparrow}$ can be exhausted by a sequence $(\tau_k)^\infty_{k=1}$ of stopping times. By passing once more to a sequence of convex combinations, still denoted by $(\widehat{\varphi}^{0,n,\uparrow})^\infty_{n=1}$, we may also assume that $(\varphi^{0,n,\uparrow}_{\tau_k})^\infty_{n=1}$ converges almost surely, for each $k \in \mathbb{N}$. Define
\begin{equation*}
\varphi^{0,\uparrow}_t=
\begin{cases}
\lim_{n \to \infty} \varphi^{0,n,\uparrow}_{\tau_k} \quad \mbox{if} \quad t=\tau_k,\quad \mbox{for some} \,\, k \in \mathbb{N}\\
\widehat{\varphi}^{0,\uparrow}_t \quad \mbox{otherwise.}
\end{cases}
\end{equation*}
This process is predictable. Indeed, there is a subset $\Omega' \subseteq \Omega$ of full measure $\mathbb{P}[\Omega']=1,$ such that $(\varphi^{0,n, \uparrow})^\infty_{n=1}$ converges pointwise to $\varphi^{0,\uparrow}$ {\it everywhere} on $\Omega' \times [0,T].$ The process $\varphi^{0, \uparrow}$ also is a.s.~non-decreasing in $t \in [0,T].$ 
We thus have found a predictable process $\varphi^{0,\uparrow}=(\varphi^{0,\uparrow}_t)_{0 \leq t \leq T}$ such that a.s.~the sequence $(\varphi^{0,n,\uparrow}_t)_{0 \leq t \leq T}$ converges to $(\varphi^{0,\uparrow}_t)_{0 \leq t \leq T}$ for {\it all} $t \in T.$

The three other cases, $\varphi^{0,\downarrow},\varphi^{1,\uparrow},$ and $\varphi^{1,\downarrow}$ are treated in an analogous way. These processes are predictable, increasing, and satisfy condition \eqref{sfc}.

Finally, define the process $(\varphi^0_t,\varphi^1_t)_{0\le t\le T}$ as 
$(\varphi^{0,\uparrow}_t -\varphi^{0,\downarrow}_t,\varphi^{1,\uparrow}_t -\varphi^{1,\downarrow}_t)_{0\le t\le T}.$ It is predictable and $M$-admissible in the numéraire-based sense \eqref{p5} as this condition passes from the process $(\varphi^n)^\infty_{n=1}$ to the limit $\varphi$. Similarly, the process $\varphi$ satisfies the self-financing condition \eqref{sfc} as the convergence of the processes $(\varphi^n)^\infty_{n=1}$ takes place, for {\it  all} $t \in [0,T].$
We thus have shown that $\A^M=\A^M_{nb}$ is closed in $L^0(\mathbb{R}^2)$. 

The closedness of $\C^M=\C^M_{nb}$ in $L^0$ is an immediate consequence.
\hfill\vrule height6pt width6pt depth0pt \medskip

\begin{remark}\label{r3.5}
\textnormal{
We have not only proved a {\it closedness} property of $\A^M$ with respect to the topology of convergence in measure. Rather we have shown a {\it convex compactness}
property (compare \cite{KZ11}, \cite{Z09}). Indeed, we have shown that, for any sequence $(\varphi^n_T)^\i_{n=1}\in\A^M,$ we can find a sequence of convex combinations which converges a.s.~to an element $\varphi_T\in \A^M.$}

\textnormal{
For later use (proof of Theorem \ref{t1.4}) we also remark that the above proof yields the following technical variant of Theorem \ref{t4.14}. 
Let $0 < \lambda_n < \lambda$ be a sequence of reals increasing to $\lambda$ and $(\varphi^n_T)^\infty_{n=1}$ be in $\A^{M, \lambda_n}_{nb}$, 
where the super-script $\lambda_n$ indicates that $\varphi^n_T$ is the terminal value of an $M$-admissible $\lambda_n$-self-financing trading strategy starting at $(0,0)$. 
If $(\varphi^n_T)^\infty_{n=1}$ converges a.s.~to $\varphi^0_T$ we may conclude that $\varphi^0_T$ is the terminal value of a strategy $\varphi^0=(\varphi^{0,0}_t, \varphi^{1,0}_t)_{0 \leq t \leq T}$ 
which is $M$-admissible and $\lambda_n$-self-financing, for each $n\in\mathbb{N}$. From \eqref{sfc} we conclude that $\varphi^0$ is $\lambda$-self-financing.}
\end{remark}

\begin{theorem}\label{t4.14.1}
(numéraire-free version) Fix $S=(S_t)_{0\le t\le T}$ and $0 < \la < 1$ as above, and suppose that $(CPS^{\la'})$ is satisfied, in the non-local sense, for each $0<\la'<\la$. For $M>0,$ the convex set $\A^M_{nf}\subseteq L^0(\ofp;\R^2)$ as well as the convex set $\C^M_{nf}\subseteq L^0(\ofp)$ are closed with respect to the topology of convergence in measure.
\end{theorem}

\begin{proof}
As in the previous proof fix $M>0$ and let $(\varphi^n_T)^\infty_{n=1}=(\varphi^{0,n}_T, \varphi^{1,n}_T)^\infty_{n=1}$ be a sequence which we now assume to be in $\A^M=\A^M_{nf}$, converging a.s~to some $\varphi_T=(\varphi^0_T, \varphi^1_T)\in L^0(\mathbb{R}^2).$ We have to show that $\varphi_T \in \A^M.$ Again we may find self-financing, admissible (in the numéraire-free sense) strategies $(\varphi^{0,n}_T, \varphi^{1,n}_t)_{0 \leq t \leq T}$ starting at $(\varphi^{0,n}_0, \varphi^{1,n}_0)=(0,0)$, with terminal values $(\varphi^{0,n}_T, \varphi^{1,n}_T).$ We now apply Th 2.4 of \cite{S13} to conclude that these processes are $M$-admissible in the numéraire-free sense \eqref{p5a}.

We then may proceed verbatim as in the above proof to construct a limiting process $\varphi=(\varphi^0_t, \varphi^1_t)_{0 \leq t \leq T}$ which is predictable, $M$-admissible (in the numéraire-free sense) and has the prescribed terminal value. 
This again shows that $\A^M=\A^M_{nf}$ and $\C^M=\C^m_{nf}$ are closed in $L^0.$
\end{proof}

\section{The proof of Theorem \ref{t1.5}}
We now apply duality theory to the sets $\A^M$ and $\C^M$.
We first deal with the numéraire-free case where we follow the lines of \cite{K99}, \cite{KS02}, \cite{CS06} and \cite{KS09}.
As above fix a c\`adl\`ag adapted price process $S=(S_t)_{0 \leq t \leq T}$ and transaction costs $0 < \lambda <1.$ We use the notation $\A_{nf}=\cup^\infty_{M=1}\A^M_{nf}$ and $\C_{nf}=\cup^\infty_{M=1}\C^M_{nf}.$
\begin{definition}\label{def a.1}
We define $\B_{nf}$ as the set of all pairs $Z_T=(Z^0_T, Z^1_T) \in L^1(\mathbb{R}^2_+)$ such that $\mathbb{E}[Z^0_T]=1$ and such that $\B_{nf}$ is polar to $\A_{nf}$, i.e.~
\begin{equation}\label{D2a}
\mathbb{E}[\varphi^0_T Z^0_T + \varphi^1_T Z^1_T] \leq 0, 
\end{equation}
for all $\varphi_T=(\varphi^0_T, \varphi^1_T) \in \A_{nf}.$

We associate to $Z_T \in \B_{nf}$ the martingale $Z$ defined by
\begin{equation}\label{D2}
Z_t= \mathbb{E}[Z_T | \F_t], \qquad 0 \leq t \leq T.
\end{equation}
\end{definition}

In \eqref{D2a} we define the expectation by requiring that the negative part of $(\varphi^0_T Z^0_T + \varphi^1_T Z^1_T)$ has to be integrable. Then \eqref{D2a} well-defines a number in $]-\infty,+\infty].$

We shall identify the elements $(Z^0_T, Z^1_T) \in \B_{nf}$ with pairs $(\widetilde{S}, Q)$ by letting 
\begin{equation}\label{p11}
\widetilde{S}_t = \frac{Z^1_t}{Z^0_t}, \quad \mbox{and} \quad \frac{dQ}{d\mathbb{P}}=Z^0_T.
\end{equation}
The random variable $Z^0_T$ may vanish on a set of positive measure. This corresponds to the fact that the probability measure $Q$ only is absolutely continuous w.r.~to $\mathbb{P}$ and not necessarily equivalent. In this case we define $\widetilde{S}_t=S_t$ where $Z^0_t$ vanishes.

We now show that $\B_{nf}$ equals precisely the set of consistent price systems $(\widetilde{S},Q)$ (in the non-local sense) where we allow $Q$ to be only absolutely continuous to $\mathbb{P}$ (in Definition \ref{1.1} we have required that $Q$ is equivalent to $\mathbb{P}$).
\begin{proposition}\label{prop4.2}
In the setting of Definition \ref{def a.1} let $Z_T \in \B_{nf}.$ Then the martingale $Z=(Z_t)_{0 \leq t \leq T}$ in \eqref{D2} satisfies
\begin{equation}\label{D3}
\widetilde{S}_t:=\frac{Z^1_t}{Z^0_t} \in [(1-\lambda)S_t, S_t], \quad 0 \leq t \leq T, \quad \mbox{a.s.} 
\end{equation}

Conversely, suppose that $Z=(Z^0_t, Z^1_t)_{0 \leq t \leq T}$ is an $\mathbb{R}^2_+$-valued $\mathbb{P}$-martingale such that $\nobreak{Z^0_0=1}$ and $\widetilde{S}_t:=\frac{Z^1_t}{Z^0_t}$ takes a.s.~on $\{Z^0_t > 0\}$ its values in $[(1-\lambda)S_t, S_t].$ Then $Z_T=(Z^0_T, Z^1_T) \in \B_{nf}.$
\end{proposition}

\begin{proof}
To show \eqref{D3} suppose that there is a $[0,T[$-valued stopping time $\tau$ such that $Q[\widetilde{S}_\tau > S_\tau] >0.$
Consider as in \eqref{155}
$$a_t = (-1, \frac{1}{S_\tau}) \mathbbm{1}_{\{\widetilde{S}_\tau > S_\tau\}}\mathbbm{1}_{\rrbracket \tau,T\rrbracket}(t), \quad 0 \leq t \leq T.$$
This is a self-financing strategy which is admissible in the numéraire-free sense (in fact, also in the numéraire-based sense) for which \eqref{D2a} yields.
\begin{align*}
\mathbb{E}_{\mathbb{P}}[(-Z^0_T + \frac{Z^1_T}{S_\tau})\mathbbm{1}_{\{\widetilde{S}_\tau > S_\tau \}}] &= \mathbb{E}_{\mathbb{P}}[\mathbb{E}_{\mathbb{P}}[(-Z^0_T + \frac{Z^1_T}{S_\tau}) \mathbbm{1}_{\{\widetilde{S}_\tau > S_\tau \}}|\F_\tau ]]\\
&=\mathbb{E}_{\mathbb{P}}[Z^0_\tau (-1 + \frac{\widetilde{S}_\tau}{S_\tau} ) \mathbbm{1}_{\{\widetilde{S}_\tau > S_\tau \}} ]\\
&=\mathbb{E}_Q [(-1 + \frac{\widetilde{S}_\tau}{S_\tau}) \mathbbm{1}_{ \{\widetilde{S}_\tau > S_\tau \}}] >0,
\end{align*}
a contradiction. In the remaining case that $Q[\widetilde{S}_T > S_T] > 0$ we consider the strategy $\nobreak{a_t=(-1,\frac{1}{S_T}) \mathbbm{1}_{\{\widetilde{S}_T >S_T\}} \mathbbm{1}_{\llbracket T \rrbracket}(t)}$ as in \eqref{B1}.

We still have to show that the case, $Q[\widetilde{S}_\tau < (1-\lambda) S_\tau] >0$, for some  stopping time $0 \leq \tau < T,$ leads to a contradiction too. As in \eqref{156} define
$$b_t=((1-\lambda) S_\tau, -1)\mathbbm{1}_{\{\widetilde{S}_\tau <(1-\lambda)S_\tau\}} \mathbbm{1}_{\rrbracket \tau,T \rrbracket}(t), \quad 0 \leq t \leq T.$$
Again this strategy is self-financing and admissible (this time only in the numéraire-free sense) and we arrive at a contradiction 
\begin{align*}
\mathbb{E}_{\mathbb{P}}[((1-\lambda)S_\tau Z^0_T - Z^1_T)\mathbbm{1}_{\{\widetilde{S}_\tau < (1-\lambda)S_\tau \}}] = \mathbb{E}_Q [((1-\lambda)S_\tau - \widetilde{S}_\tau ) \mathbbm{1}_{\{\widetilde{S}_\tau < (1-\lambda)S_\tau \}}] >0.
\end{align*}
The case $Q[\widetilde{S}_T < (1-\lambda)S_T]$ is dealt by considering \eqref{B1a} similarly as above.
This shows the first part of the proposition.

As regards the second part, fix a martingale $Z=(Z^0_t, Z^1_t)_{0 \leq t \leq T}$ with the properties stated there and let $(\widetilde{S},Q)$ be defined by \eqref{p11}. For every self-financing trading strategy $\varphi=(\varphi^0_t, \varphi^1_t)_{0 \leq t \leq T}$, starting at $(\varphi^0_0, \varphi^1_0)=(0,0)$ and being $M$-admissible in the numéraire-free sense we deduce from Proposition 2.3 and the subsequent remark in \cite{S13} that
$\widetilde{V}_t:=\varphi^0_t + \varphi^1_t \widetilde{S}_t$ is an optional strong super-martingale under $Q$ (see \cite{S13}, Def.~1.5, for a definition). This gives the desired inequality
$$0 = \widetilde{V}_0 \geq \mathbb{E}_Q[\widetilde{V}_T]= \mathbb{E}_{\mathbb{P}}[ \varphi^0_T Z^0_T + \varphi^1_T Z^1_T].$$
The proof of Proposition \ref{prop4.2} is now complete.
\end{proof}

In order to obtain a proof of Th. \ref{t1.5} we still need a version of the bipolar theorem for $L^0.$ We first recall the bipolar theorem in the one-dimensional setting as obtained in \cite{BS99}. For a subset $A \subseteq L^0(\mathbb{R}_+)$ we define its polar in $L^1(\mathbb{R}_+)$ by
$$A^0=\{g \in L^1(\mathbb{R}_+): \mathbb{E}[fg]\leq 1\}.$$
The bipolar theorem in \cite{BS99} states that $f \in L^0(\mathbb{R}_+)$ belongs to the closed (w.r. to convergence in measure), convex, solid hull of $A$ if and only if
$$\mathbb{E}[fg]\leq 1, \quad \mbox{for all} \quad g \in A^0.$$

We need the multi-dimensional version of this result established in (\cite{KS09}, Th.~5.5.3) which applies to the cone $\A_{nf}$ in $L^0(\mathbb{R}^2).$

While in the one-dimensional setting considered in \cite{BS99} there is just one natural order structure of $L^0(\mathbb{R})$, in the two-dimensional setting the situation is more complicated (see \cite{BM03}). We define a partial order on $L^0(\mathbb{R}^2)$ by letting $\varphi_T=(\varphi^0_T, \varphi^1_T) \succeq \psi_T=(\psi^0_T, \psi^1_T)$ if the difference $\varphi_T - \psi_T$ may be liquidated to the zero-portfolio, i.e. $V_T(\varphi_T - \psi_T) \geq 0.$ This partial order is designed in such a way that, for $\varphi_T \in \A^M_{nf}$, we have that $\varphi_T  \succeq(-M, -M).$

Following \cite{KS09} we say that a sequence $(\varphi^n_T)^\infty_{n=1}$ in $L^0(\mathbb{R}^2)$ {\it Fatou-converges} to $\varphi_T \in L^0(\mathbb{R}^2)$ if there is $M>0$ such that each $\varphi^n_T$ dominates $(-M,-M)$ and $(\varphi^n_T)^\infty_{n=1}$ converges a.s.~to $\varphi_T.$

By (a version of) Fatou's lemma this convergence implies that, for each $Z_T=(Z^0_T, Z^1_T) \in \B_{nf}$, 
$$\liminf_{n \to \infty} \langle \varphi^n_T, Z_T \rangle:= \liminf_{n \to \infty} \mathbb{E}[\varphi^{0,n}_T Z^0_T + \varphi^{1,n}_T Z^1_T] \geq \mathbb{E}[\varphi^0_T Z^0_T + \varphi^1_T Z^1_T] = \langle \varphi_T, Z_T\rangle,$$
as $\varphi^{0,n}_T Z^0_T + \varphi^{1,n}_T Z^1_t \geq -M(Z^0_T + Z^1_T)$ and the latter function is $\mathbb{P}$-integrable.
\vskip10pt
Denote by $\A^b_{nf}$ the set of bounded elements in $\A_{nf}$, i.e.~$\A^b_{nf}=\A_{nf} \cap L^\infty(\mathbb{R}^2).$ It is straightforward to deduce from Theorem \ref{t4.14.1} that under the assumption of Theorem \ref{t1.5} the following properties are satisfied.

\begin{itemize}
\item [$(i)$] $\A_{nf}$ is Fatou-closed, i.e.~contains all limits of its Fatou-convergent sequences.
\item [$(ii)$] $\A^b_{nf}$ is Fatou-dense in $\A_{nf}$, i.e.~for $\varphi_T \in \A_{nf}$, there is a sequence $(\varphi^n_T)^\infty_{n=1} \in \A^b_{nf}$ which Fatou-converges to $\varphi_T.$
\item [$(iii)$] $\A^b_{nf}$ contains the negative orthant $-L^\infty(\mathbb{R}^2_+).$
\end{itemize}

Define the polar of $\A_{nf}$ by
$$\A^0_{nf}=\{Z_T=(Z^0_T, Z^1_T) \in L^1(\mathbb{R}^2):\langle\varphi_T, Z_T\rangle \leq 1\}.$$

As $\A_{nf}$ is a cone we may equivalently write
$$\A^0_{nf}=\{(Z_T=(Z^0_T, Z^1_T) \in L^1(\mathbb{R}^2):\langle\varphi_T, Z_T\rangle \leq 0\}.$$

Proposition \ref{prop4.2} states that $\A^0_{nf}$ equals the cone generated by $\B_{nf}.$

It is shown in (\cite{KS09}, Th.~5.5.3) that the three properties above imply that, for the set $\A_{nf}$ in $L^0(\mathbb{R}^2)$ which satisfies $(i)$, $(ii)$ and $(iii)$, the bipolar theorem holds true, i.e.~an element $X_T=(X^0_T, X^1_T)\in L^0(\mathbb{R}^2)$ such that $X_T \succeq (-M, -M), \ \mbox{for some} \ M>0,$ is in $\A_{nf}$ if and only if, 
\begin{equation}\label{P3}
\langle X_T, Z_T \rangle:= \mathbb{E}[X^0_T Z^0_T + X^1_T Z^1_T] \leq 0, \quad \mbox{for every} \quad Z_T \in \A^0_{nf},
\end{equation}

By normalising, it is equivalent to require the validity of \eqref{P3} for all $Z_T \in \B_{nf}.$
\vskip10pt
We thus have assembled all the ingredients for a proof of the numéraire-free version of the super-hedging theorem.
\vskip10pt
{\bf Proof of Theorem \ref{t1.5}:}
The above discussion actually yields the following two-dimensional result which is more general than the one-dimensional statement of Theorem \ref{t1.5}. Under the hypotheses of Theorem \ref{t1.5} consider a contingent claim $X_T=(X^0_T, X^1_T)$ which delivers $X^0_T$ many bonds and $X^1_T$ many stocks at time $T$.
Then there is a self-financing, admissible (in the numéraire-free sense) strategy $\varphi$, starting with $(\varphi^0_0, \varphi^1_0)=(0,0)$ and ending with $(\varphi^0_T, \varphi^1_T)=(X^0_T, X^1_T)$ if and only if

\begin{equation}\label{P4}
\langle X_T, Z_T \rangle= \mathbb{E}_{\mathbb{P}}[X^0_T Z^0_T + X^1_T Z^1_T] = \mathbb{E}_Q[X^0_T + X^1_T \widetilde{S}_t] \leq 0,
\end{equation}
for every $Z_T \in \B_{nf}.$ This is just statement \eqref{P3}, where $(\widetilde{S},Q)$ is given by \eqref{p11}, i.e.~$Q$ is a probability measure, absolutely continuous w.r.~to $\mathbb{P}$, and $\widetilde{S}$ is a (true) $Q$-martingale taking values in $[(1-\lambda) S,S].$

We still need two observations. In \eqref{P4} we may equivalently assume that the probability measure $Q$ is actually {\it equivalent} to $\mathbb{P}$, i.e.~the corresponding martingale $Z$ satisfies $Z^0_T > 0$ almost surely. Indeed, fix $Z_T \in \B_{nf}$ as in \eqref{P4}. By assumption $(CPS^\lambda)$ (in the non-local sense) there is {\it some} $\overline{Z}_T \in \B_{nf}$ verifying $\overline{Z}^0_T > 0$ almost surely. Note that $\langle X_T, \overline{Z}_T \rangle$ takes a finite value. For $0 < \mu < 1$ the convex combination $\mu \overline{Z}_T + (1-\mu) Z_T$ is in $\B_{nf}$ and still satisfies the strict positivity condition. Sending $\mu$ to zero we see that in \eqref{P4} we may assume w.l.g.~that $Z^0_T$ is a.s.~strictly positive.

A second remark pertains to the initial endowment $(\varphi^0_0, \varphi^1_0)$ which in \eqref{P4} we have normalised to $(0,0).$ If we replace $(0,0)$ by an arbitrary pair $(X^0_0, X^1_0) \in \mathbb{R}^2$ then \eqref{P4} trivially translates to the equivalence of the following two statements for a contingent claim $(X^0_T, X^1_T)$ verifying 
$$V_T(X^0_T, X^1_T) \geq -M(1+S_T), \quad \mbox{for some} \quad M>0.$$

\begin{itemize}
\item [$(i)$] There is a self-financing, admissible (in the numéraire-free sense) trading strategy $\varphi=(\varphi^0_t, \varphi^1_t)_{0 \leq t \leq T}$ such that 
$$\varphi_0=(X^0_0, X^1_0) \quad \mbox{and} \quad \varphi_T= (X^0_T, X^1_T)$$
\item [$(ii)$] For every consistent price system, i.e.~each probability measure $Q$, equivalent to $\mathbb{P}$ such that there is a martingale $\widetilde{S}$ under $Q$, taking its values in the bid-ask spread $[(1-\lambda) S,S],$ we have
$$\mathbb{E}_Q[(X^0_T-X^0_0)+(X^1_T - X^1_0)\widetilde{S}_T] \leq 0.$$
\end{itemize}
Specialising to the case where $X^0_T$ and $X^1_T$ is equal to zero we obtain the assertion of  Theorem \ref{t1.5}.
\hfill\vrule height6pt width6pt depth0pt \medskip

\section{The proof of Theorem \ref{t1.4}}
We now deduce the numéraire-based super-replication theorem from its numéraire-free counterpart.

$(i) \Rightarrow (ii)$ This is the easy implication. Suppose that $X_T$ and $\varphi=(\varphi^0_t, \varphi^1_t)_{0 \leq t \leq T}$ are given as in $(i)$ of Theorem \ref{t1.4}. Let $(\widetilde{S}, Q)$ be a consistent local price system. By Proposition 1.6 in \cite{S13}, the process $\widetilde{V}_t = (\varphi^0_t + \varphi^1_t \widetilde{S}_t)_{0 \leq t \leq T}$ is an optional strong super-martingale under $Q$ which implies \eqref{zp4}.

$(ii) \Rightarrow (i)$ Conversely, let $X_T \geq-M$ be as in the statement of Theorem \ref{t1.4} and suppose that $(ii)$ holds true. Define the $[0,T] \cup \{\infty\}$-valued stopping time $\tau_n$ by 
$$\tau_n = \inf \{ t : S_t \geq n\}.$$
Also define 
\begin{equation*}
X^n_T=
\begin{cases}
X_T, \quad &\mbox{on} \quad \{\tau_n = \infty \},\\
-M,\quad &\mbox{on} \quad \{\tau_n \leq T\},
\end{cases}
\end{equation*}
so that $(X^n_T)^\infty_{n=1}$ is $\F_{\tau_n}$-measurable and increases a.s.~to $X_T$.

Let $0 < \lambda_n < \lambda$ be a sequence of reals increasing to $\lambda.$

For fixed $n \in \mathbb{N}$ we may apply Theorem \ref{t1.5} to the stopped process $S^{\tau_n}$, the random variable $X^n_T$ and transaction costs $\lambda_n$. To verify that the conditions of Theorem \ref{t1.5} are indeed satisfied note that under the hypotheses of Theorem \ref{t1.4}, for every $0 < \lambda' < 1$, condition $(CPS^{\lambda'})$ is satisfied for $S$ in a local sense and therefore -- by stopping -- also for $S^{\tau_n}$. By Proposition \ref{prop6.1} below we conclude that $(CPS^{\lambda'})$ is, in fact, satisfied in a non-local sense for the process $S^{\tau_n}$ as required by Theorem \ref{t1.5}.

Next we show that condition $(ii)$ of Theorem \ref{t1.5} is satisfied for the process $S^{\tau_n}$ and transaction costs $\lambda_n$. Indeed, fix $n$ and let $Q\sim \mathbb{P}$ be such that there is a $Q$-martingale $\widetilde{S}=(\widetilde{S}_t)_{0 \leq t \leq T}$ taking its values in $[(1-\lambda_n) S^{\tau_n}, S^{\tau_n}]$ and associate the martingales $Z^0, Z^1$ to $(Q, \widetilde{S})$.

We may concatenate this $\lambda_n$-consistent price systems for $S^{\tau_n}$ to a $\lambda$-consistent local price system $\overline{Z}=(\overline{Z}^0_t, \overline{Z}^1_t)_{0 \leq t \leq T}$ for the process $S$.

Here are the details. Fix $0 < \lambda' < \frac{\lambda - \lambda_n}{2}.$ By the assumption of Theorem \ref{t1.4} there is a $\lambda'$-consistent local price system $\check{Z}=(\check{Z}^0_t, \check{Z}^1_t)_{0 \leq t \leq T}$ for $S$. Define $\overline{Z}$ by
\begin{align*}
\overline{Z}^0_t&=
\begin{cases}
Z^0_t, \hspace{14mm} &0  \leq t \leq \tau_n\\
\check{Z}^0_t \frac{Z^0_\tau}{\check{Z}^0_\tau}, \hspace{14mm} &\tau_n \leq \tau \leq T,\\
\end{cases}\\
\overline{Z}^1_t&=
\begin{cases}
(1-\lambda')Z^1_t, &0 \leq t \leq \tau_n\\
(1-\lambda')\check{Z}^1_t \frac{Z^1_\tau}{\check{Z}^1_\tau}, &\tau_n \leq \tau \leq T.
\end{cases}
\end{align*}

Clearly, $\overline{Z}^0$ (resp.$\overline{Z}^1$) is an $\mathbb{R}_+$-valued martingale (resp.~local martingale) under $\mathbb{P}$ and $\frac{d\overline{Q}}{d\mathbb{P}}=\overline{Z}^0_T$ defines a probability measure on $\F$ equivalent to $\mathbb{P}$. To show that $\frac{\overline{Z}^1}{\overline{Z}^0}$ takes its values in $[(1-\lambda)S,S]$ note that, for $0 \leq t \leq \tau_n$, the quotient $\frac{\overline{Z}^1_t}{\overline{Z}^0_t}$ lies in $[(1-\lambda_n)(1-\lambda')S_t, (1-\lambda')S_t].$ For $\tau_n \leq t \leq T$ we still obtain that $\frac{\overline{Z}^1_t}{\overline{Z}^0_t}$ lies in $[(1-\lambda_n)(1-\lambda')^2 S_t, \frac{1-\lambda'}{1-\lambda'}S_t]$ which is contained in $[(1-\lambda)S_t, S_t]$ as $\lambda' < \frac{\lambda-\lambda_n}{2}.$ By assumption $(ii)$ of Theorem \ref{t1.4} we conclude that
\begin{align*}
\mathbb{E}_Q[X^n_T] &= \mathbb{E}_{\overline{Q}}[X^n_T]\\
& \leq \mathbb{E}_{\overline{Q}}[X_T] \leq X_0.
\end{align*}
Hence we may apply Theorem \ref{t1.5} to conclude that there is a $\lambda_n$-self-financing trading strategy $\varphi^n=(\varphi^{0,n}_t, \varphi^{1,n}_t)_{0 \leq t \leq T}$ for $S$ such that 
$\varphi^n_0= (X_0,0)$ and $\varphi^n_T=\varphi^n_{\tau_n}=(X^n_T,0)$ and which is $M$-admissible in the sense of \eqref{D4}. Applying Theorem 2.5 in \cite{S13} to the case $x=M$ and $y=0$ we may conclude that each $\varphi^n$ is, in fact, $M$-admissible in the sense of \eqref{D3.1}. Finally, we apply Theorem \ref{t4.14} and the subsequent Remark \ref{r3.5}, which yields the desired self-financing trading strategy $\varphi$ as a limit of $(\varphi^n)^\infty_{n=1}$. This strategy $\varphi$ has the properties stated in Theorem \ref{t1.4} $(i).$
\hfill\vrule height6pt width6pt depth0pt \medskip

\section{Appendix}
The following proposition seems to be a well-known folklore type result. As we are unable to give a reference we provide a proof.

\begin{proposition}\label{prop6.1}
Let $(X_t)_{0 \leq t \leq T}$ be an $\mathbb{R}_+$-valued local martingale, $\tau$ a stopping time, and $C>0$ a constant such that $X_t \leq C$, for $0 \leq t < \tau.$ Then the stopped process $X^\tau$ is a martingale.
\end{proposition}

\begin{proof}
It follows from Fatou's lemma and the boundedness from below that $X$ is a super-martingale. Hence it will suffice to show that 
\begin{equation}\label{A1}
\mathbb{E}[X_\tau]=X_0.
\end{equation}
By hypothesis there is a sequence $(\sigma_k)^\infty_{k=1}$ of $[0,T] \cup \{\infty\}$-valued stopping times, increasing to $\infty$, such that 
$$\mathbb{E}[X_{\sigma_k \wedge \tau}]=X_0, \quad \mbox{for} \quad k \geq 1.$$
As $\lim_{k \to \infty} \mathbb{P} [\sigma_k < \tau] = 0$ and $X_{\sigma_k}$ is bounded by $C$ on $\{\sigma_k < \tau\}$ we obtain from the monotone convergence theorem:
\begin{align*}
X_0&=\lim_{k \to \infty} \mathbb{E}[X_\tau \mathbbm{1}_{\{\sigma_k \geq \tau \}} + X_{\sigma_k} \mathbbm{1}_{\{\sigma_k < \tau \}}]\\
&= \mathbb{E}[X_\tau].
\end{align*}
This gives \eqref{A1}.
\end{proof}

\vskip20pt
\begin{acka}
We thank Irene Klein for her insistency on the topic as well as fruitful discussions on the proof of  Theorem \ref{t1.4}, and Christoph Czichowsky for his advise and careful reading of the paper.
\end{acka}

\bibliographystyle{plain}

\end{document}